\documentclass[reqno]{amsart}
\usepackage{amsmath, amssymb, amsthm, epsfig, float, enumitem}
\usepackage{hyperref, latexsym}
\usepackage{url}
\usepackage[mathscr]{euscript}

\usepackage{color}
\usepackage{fullpage} 
\usepackage{setspace}
\usepackage{mathtools}

\usepackage{tikz}
\usetikzlibrary{arrows}

\def\today{\ifcase\month\or
  January\or February\or March\or April\or May\or June\or=
  July\or August\or September\or October\or November\or December\fi
  \space\number\day, \number\year}

 \newtheorem{theorem}[equation]{Theorem}
  
 \newtheorem{lemma}[equation]{Lemma}
 \newtheorem{proposition}[equation]{Proposition}
 \newtheorem{corollary}[equation]{Corollary}
 \theoremstyle{definition}
 
 \newtheorem{example}[equation]{Example}
 \theoremstyle{remark}

\newtheorem{question}{Question}

 \newcommand{\R}{\mathbb{R}}
  
 \newcommand{\N}{\mathbb{N}}

\newcommand{\intav}[1]{\mathchoice {\mathop{\vrule width 6pt height 3 pt depth  -2.5pt
\kern -8pt \intop}\nolimits_{\kern -6pt#1}} {\mathop{\vrule width
5pt height 3  pt depth -2.6pt \kern -6pt \intop}\nolimits_{#1}}
{\mathop{\vrule width 5pt height 3 pt depth -2.6pt \kern -6pt
\intop}\nolimits_{#1}} {\mathop{\vrule width 5pt height 3 pt depth
-2.6pt \kern -6pt \intop}\nolimits_{#1}}}

\begin{document}

\title[The maximal function of the Devil's staircase is absolutely continuous]{The maximal function of the Devil's staircase \\ is absolutely continuous}
\author[Gonz\'{a}lez-Riquelme and Kosz]{Cristian Gonz\'{a}lez-Riquelme and Dariusz Kosz}
\subjclass[2010]{26A45, 42B25, 39A12, 46E35, 46E39, 05C12.}
\keywords{Maximal operator, continuity, bounded variation.}
\address{ICTP -- The Abdus Salam International Centre for Theoretical Physics,
Strada Costiera 11,
I - 34151 Trieste, Italy}
\email{cgonzale@ictp.it}
\address{BCAM -- Basque Center for Applied Mathematics, 48009 Bilbao, Spain
	\& Wroc\l aw University of Science and Technology, 50-370
	Wroc\l aw, Poland. 
	}
\email{dariusz.kosz@pwr.edu.pl}

\allowdisplaybreaks
\numberwithin{equation}{section}

\maketitle

\begin{abstract}
We study the problem of whether the centered Hardy--Littlewood maximal function of a singular function is absolutely continuous. For a parameter $d \in (0,1)$ and a closed set $E\subset [0,1]$, let $\mu$ be a $d$-Ahlfors regular measure associated with $E$. We prove that for the cumulative distribution function $f(x)=\mu([0,x])$ its maximal function $Mf$ is absolutely continuous. We then adapt our method to the multiparameter case and show that the same is true in the positive cone defined by these functions, i.e., for functions of the form $f(x)=\sum_{i=1}^{n}\mu_i([0,x])$ where $\{\mu_i\}_{i=1}^{n}$ is any collection of $d_i$-Ahlfors regular measures, $d_i \in (0,1)$, associated with closed sets $E_i\subset [0,1]$. This provides the first improvement of regularity for the classical centered maximal operator, and can be seen as a partial analogue of the result of Aldaz and Pérez Lázaro about the uncentered maximal operator.    
\end{abstract}

\section{Introduction} \label{S1}
The study of regularity properties for maximal operators was started by Kinnunen in \cite{Kinnunen1997}. It was established there that the centered Hardy--Littlewood maximal operator is bounded from $W^{1,p}(\mathbb{R})$ to itself if $p \in (1,\infty)$. This operator is defined as follows: for $f\in L^{1}_{\text{loc}}(\mathbb{R})$, and $m$ the standard Lebesgue measure on $\R$, we let 
\begin{align*}
Mf(x) \coloneqq \sup_{r \in (0,\infty)} \, \intav{B(x,r)}|f|,
\quad \text{where} \quad \intav{B}|f| \coloneqq \frac{\int_{B}|f|}{m(B)}. 
\end{align*}
The said boundedness holds also for the uncentered Hardy--Littlewood maximal operator $\widetilde{M}$, that is, the operator defined with the aid of all balls containing $x$, not just the ones centered at it. 

After Kinnunen's seminal work, many attempts to understand when a given maximal operator preserves the regularity of Sobolev or $BV$ functions were made. In particular, the endpoint case $p=1$ was very widely discussed. First, it was proved by Tanaka \cite{Tanaka2002} that the map $f\mapsto (\widetilde{M}f)'$ is bounded from $W^{1,1}(\mathbb{R})$ to $L^{1}(\mathbb{R})$. Later, in the remarkable article \cite{Kurka2010}, Kurka established the analogous result for the centered operator. In the uncentered case, many -- also multidimensional -- generalizations of these results were made, see for instance \cite{Luiro2017, charweigt}. In particular, in \cite{weigt2021variation} the case of the uncentered operator using multidimensional cubes was investigated in full generality.

Although showing that maximal operators can preserve certain types of regularity is still a lively topic with many open problems and lines of research to study, in this manuscript we aim at exploring the following, more ambitious question:
\begin{question}\label{Q1}
Given a maximal operator $\mathcal M$, for what class of functions $f$ can one observe a substantial improvement between $f$ and $\mathcal M f$ in terms of their regularity?
\end{question}
\noindent This question was considered by Aldaz and Pérez Lázaro. In \cite{AP2007}, they established that $\widetilde{M}f \in AC(\R)$ whenever $f\in BV(\mathbb{R})$\footnote{Here $AC(\R)$ and $BV(\R)$ are the spaces of absolutely continuous functions and functions of bounded variation, respectively.}. Also, improvement of regularity was investigated in \cite{CM2015, HKKT2015, RSW} (for fractional maximal operators) and in \cite{lahti} (for the uncentered Hardy--Littlewood maximal operator in higher dimensions).  

The improvement of regularity established by Aldaz and Pérez Lázaro was a crucial tool in the work \cite{GRK} by the authors, where for the map $f\mapsto \widetilde{M}f$ its continuity from $BV(\mathbb{R})$ to itself was proven. A very interesting open problem following this result is to examine the continuity of the map $f\mapsto Mf$ from $BV(\mathbb{R})$ to itself. In fact, this was asked by Carneiro, Madrid and Pierce \cite{CMP2017} as part of their {\it {one-dimensional endpoint continuity program}}, and among the problems mentioned there it is the only one that remains open (see \cite{grcenteredcontinuity} for further discussion). 

The relation between smoothing properties of maximal operators and the said continuity on $BV(\mathbb{R})$ is one of the motivations to study the following version of Question~\ref{Q1}:
\begin{question}\label{Q2}
For which functions $f\in BV(\mathbb{R})$ do we have $Mf$ absolutely continuous?
\end{question}
\noindent Recall that in the centered case for $f \in BV(\R)$ its maximal function is not necessarily continuous\footnote{Consider, for example, the Heaviside step function.}. On the other hand, for $f \in BV(\R) \cap AC(\R)$ the absolute continuity of $Mf$ is known since \cite[Corollary~1.3]{Kurka2010}. Thus, the first natural set to study is $BV(\R) \cap C(\R) \setminus AC(\R)$. The canonical example of this set is the ternary Cantor function which is analyzed in detail in Section~\ref{S3}. In this case, we obtain the absolute continuity of the maximal function by making use of the arithmetic structure of the Cantor set. However, this approach is not flexible enough to study more general functions. A class of function that is of particular interest for us consists of the ones associated with $d$-Ahlfors regular measures. For a closed set $A \in [0,1]$ and a positive Borel measure $\eta$, we say that $A$ is $d$-Ahlfors regular with respect to $\eta$, $d \in (0,1)$, if there exists a constant $C \in (1, \infty)$ such that, for every $x\in A$ and $r \in (0,1)$, we have 
\[
C^{-1}r^{d}\le \eta([x-r,x+r])\le Cr^{d}.
\]
For brevity, we then also call $\eta$ a $d$-Ahlfors regular measure associated with $A$. In particular, it is known that the Hausdorff measure over the ternary Cantor set is $d$-Ahlfors regular for $d=\frac{\log(2)}{\log(3)}$. In this manuscript, we study Question \ref{Q2} for the positive cone defined by linear combinations of cumulative distribution functions corresponding to these measures.     
In order to deal with this more general case, we develop a new method, similar in spirit to the one used in the Cantor ternary case, but with arithmetic-type arguments properly adjusted. This leads to the following theorem, which is the main result of the article.

\begin{theorem}\label{maintheoremmultiple}
Given $n \in \N$, let $f(x)=\sum_{i=1}^{n}\mu_i([0,x])$ where $\{\mu_i\}_{i=1}^{n}$ is a collection of $d_i$-Ahlfors regular measures, $d_i \in (0,1)$, associated with closed sets $E_i\subset [0,1]$. Then $Mf$ is absolutely continuous. 
\end{theorem}

\noindent Theorem~\ref{maintheoremmultiple} can be seen as a partial analogue of the result of Aldaz and Pérez Lázaro \cite[Theorem~2.5]{AP2007} where they studied Question \ref{Q2} for $\widetilde{M}$. We emphasize that, in questions about regularity, centered operators are usually the more challenging ones. Also, we would like to mention that Question \ref{Q2} for general singular monotonic continuous functions remains a very interesting open problem. 

The rest of the paper is organized as follows. In Section~\ref{S2} we collect several basic facts about maximal functions and discuss a very helpful reduction that can be made in the proof of Theorem~\ref{maintheoremmultiple}. Section~\ref{S3} is devoted to studying the special case of the Cantor function. In Section~\ref{S4} we prove Theorem~\ref{maintheoremmultiple} for $n=1$. According to the reduction stated in Section~\ref{S2}, the proof is based on estimating the Lebesgue measure of the image of the set $\{x\in [0,1]; f(x)=Mf(x)\}$. A key insight here comes from the observation that this set resembles a fractal. We construct a recursive geometric-type process which plays a crucial role in this analysis. 
Finally, in Section~\ref{S5} we prove the main theorem in full generality. Here we use an inductive argument which combines the result for $n=1$ with some refinements of the techniques used in this case. In particular, the interplay between the measures of different dimensions needs to be treated carefully. 

\subsection{Acknowledgements} 
The authors are grateful to Emanuel Carneiro for inspiring discussions on the topic of this project. The first author is also thankful to Andrea Olivo for helpful discussions.
The second author was supported by the Basque Government through the BERC 2018-2021 program, by the Spanish State Research Agency through BCAM Severo Ochoa excellence accreditation SEV-2017-2018, and by the Foundation for Polish Science through the START Scholarship.
\section{Preliminaries and main reduction} \label{S2}
 
Let us recall some known facts that will be useful for our purposes.
\begin{lemma}\label{continuous}
If $g\in BV(\R) \cap C(\R)$, then $Mg \in C(\R)$. 
\end{lemma}
\begin{proof}
This follows from \cite[Lemma~7.4]{Kurka2010}.
\end{proof}
\begin{lemma}\label{locallyl}
 If $g\in BV(\R) \cap C(\mathbb{R})$, then the set $D_g \coloneqq \{x\in \mathbb{R}; Mg(x)>|g|(x)\}$ is open and $Mg$ is locally Lipschitz in $D_g$. 
\end{lemma}
\begin{proof}
This follows as in \cite[Lemma~7.1]{Kurka2010}
\end{proof}
\noindent Also, let us recall the classical Banach--Zarecki theorem.
\begin{lemma}[Banach--Zarecki theorem] \label{banackzarechi}
 A function $g \colon [0,1]\to \mathbb{R}$ is absolutely continuous if and only if it is continuous, of bounded variation, and maps $m$-null sets to $m$-null sets. Also, for $g \colon \mathbb{R}\to \mathbb{R}$ the last three properties together imply absolute continuity. 
\end{lemma}
\noindent Combining these results gives the following proposition reducing the problem stated in Theorem~\ref{maintheoremmultiple}.
\begin{proposition}\label{justluzinproperty}
If $g\in BV(\R) \cap C(\R)$ and $m(g(\{x\in \mathbb{R}; Mg(x)=|g|(x)\}))=0$, then $Mg \in AC(\R)$.       
\end{proposition}
\begin{proof}
The function $Mg$ is continuous by Lemma~\ref{continuous} and of bounded variation by Kurka's theorem \cite{Kurka2010}. Therefore, by Lemma~\ref{banackzarechi} it is enough to prove that for any $m$-null set $N$ we have $m(Mg(N))=0$. Let us write $N=N_1\cup N_2$ where $N_1 \coloneqq \{x\in N; Mg(x)=|g|(x)\}$ and $N_2 \coloneqq \{x\in N; Mg(x)>|g|(x)\}$. We have $m(Mg(N_2))=0$ because $N_2 \subset D_g$ and $Mg$ is locally Lipschitz in $D_g$ by Lemma~\ref{locallyl}. Therefore, since $m(Mg(N_1))\le m(Mg(\{x\in \mathbb{R};Mg(x)=|g|(x)\}))=0$, the thesis follows.      
\end{proof}




\section{Toy model: ternary Cantor function} \label{S3}
We first consider a particular example. Let $h$ be the ternary Cantor function associated with the ternary Cantor measure $\mu_{C}$ supported on the Cantor set $\mathcal{C}$. Precisely, one can define $h$ as follows:
\[
h(x) \coloneqq \begin{cases}
0, & \quad x \in (-\infty,0),\\
\sum_{k=1}^\infty \frac{a_k}{2^k}, &\quad x = \sum_{k=1}^\infty \frac{2a_k}{3^k} \in \mathcal C \text{ for some } (a_k)_{k=1}^\infty \in \{0,1\}^{\mathbb N},\\
\sup_{y \in \mathcal C \cap [0,x]} h(y), & \quad x \in [0,1] \setminus \mathcal C, \\
1, & \quad x \in (1,\infty).
\end{cases} 
\]
Also, denote $C_h \coloneqq \{ x \in \mathbb R; Mh(x) = h(x)\}$.
In this case, thanks to the very regular structure of $\mathcal C$, one can show $m(h (C_{h})) = 0$ straightforward. Although the general case, that is treated in the next sections, requires much deeper analysis, some ideas behind the proof can be found in this particular example. 
 
\begin{example}\label{E1}
	For the Cantor function $h$ one has $m(h(C_h)) = 0$.
\end{example}

\begin{center}
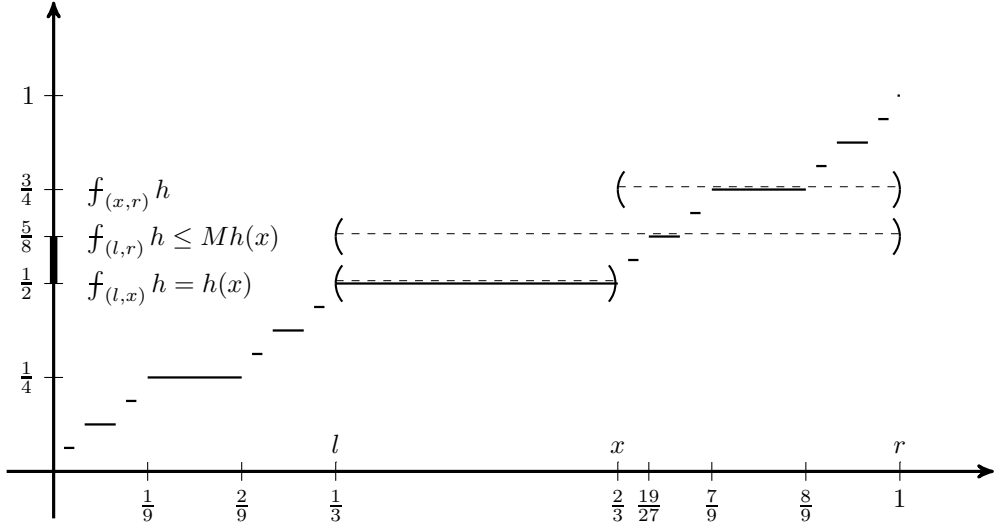
\begin{figure}[H]
\begin{tikzpicture}
[
scale=1.25,
axis/.style={very thick,->,>=stealth'},
important line/.style={thick},
dashed line/.style={dashed, thin}
]

\draw[axis] (-0.5,0)  -- (10,0) node(xline)[right]
{};
\draw[axis] (0,-0.5) -- (0,5) node(yline)[above] {};

\draw (-0.1,1) node[left] {$\frac{1}{4}$} -- (0.1,1);
\draw (-0.1,2) node[left] {$\frac{1}{2}$} -- (0.1,2);
\draw (-0.1,3) node[left] {$\frac{3}{4}$} -- (0.1,3);
\draw (-0.1,4) node[left] {$1$} -- (0.1,4);

\draw (-0.1,2.5) node[left] {$\frac{5}{8}$} -- (0.1,2.5);

\draw (1,-0.1) node[below] {$\frac{1}{9}$} -- (1,0.1);
\draw (2,-0.1) node[below] {$\frac{2}{9}$} -- (2,0.1);
\draw (3,-0.1) node[below] {$\frac{1}{3}$} -- (3,0.1);

\draw (6,-0.1) node[below] {$\frac{2}{3}$} -- (6,0.1);
\draw (7,-0.1) node[below] {$\frac{7}{9}$} -- (7,0.1);
\draw (8,-0.1) node[below] {$\frac{8}{9}$} -- (8,0.1);
\draw (9,-0.1) node[below] {$1$} -- (9,0.1);

\draw (3,0.1) node[above] {$l$} -- (3,0.1);
\draw (6,0.1) node[above] {$x$} -- (6,0.1);
\draw (9,0.1) node[above] {$r$} -- (9,0.1);

\draw (0.0,1.95) node[right] {$\quad \intav{(l,x)}h = h(x)$} -- (0.0,1.95);
\draw (0.0,2.45) node[right] {$\quad \intav{(l,r)}h \leq Mh(x)$} -- (0.0,2.45);
\draw (0.0,2.95) node[right] {$\quad \intav{(x,r)}h$} -- (0.0,2.95);

\draw (6.33,-0.1) node[below] {$\frac{19}{27}$} -- (6.33,0.1);

\draw[important line, line width=0.3mm] (3,2) -- (6,2);

\draw[important line, line width=0.3mm] (1,1) -- (2,1);
\draw[important line, line width=0.3mm] (7,3) -- (8,3);

\draw[important line, line width=0.3mm] (0.33,0.5) -- (0.66,0.5);
\draw[important line, line width=0.3mm] (2.33,1.5) -- (2.66,1.5);
\draw[important line, line width=0.3mm] (6.33,2.5) -- (6.66,2.5);
\draw[important line, line width=0.3mm] (8.33,3.5) -- (8.66,3.5);

\draw[important line, line width=0.3mm] (0.11,0.25) -- (0.22,0.25);
\draw[important line, line width=0.3mm] (0.77,0.75) -- (0.88,0.75);
\draw[important line, line width=0.3mm] (2.11,1.25) -- (2.22,1.25);
\draw[important line, line width=0.3mm] (2.77,1.75) -- (2.88,1.75);
\draw[important line, line width=0.3mm] (6.11,2.25) -- (6.22,2.25);
\draw[important line, line width=0.3mm] (6.77,2.75) -- (6.88,2.75);
\draw[important line, line width=0.3mm] (8.11,3.25) -- (8.22,3.25);
\draw[important line, line width=0.3mm] (8.77,3.75) -- (8.88,3.75);

\draw[important line, line width=0.3mm] (8.975,4) -- (9,4);

\draw[dashed line, line width=0.1mm] (3,2.03) -- (6,2.03);
\draw[dashed line, line width=0.1mm] (3,2.53) -- (9,2.53);
\draw[dashed line, line width=0.1mm] (6,3.03) -- (9,3.03);


\draw[important line, line width=1mm] (0,2) -- (0,2.5);

\draw[black, thick] (3,2) arc (180:140:0.3);
\draw[black, thick] (3,2) arc (180:220:0.3);

\draw[black, thick] (5.975,2) arc (0:40:0.3);
\draw[black, thick] (5.975,2) arc (0:-40:0.3);

\draw[black, thick] (6,3) arc (180:140:0.3);
\draw[black, thick] (6,3) arc (180:220:0.3);

\draw[black, thick] (9,3) arc (0:40:0.3);
\draw[black, thick] (9,3) arc (0:-40:0.3);

\draw[black, thick] (3,2.5) arc (180:140:0.3);
\draw[black, thick] (3,2.5) arc (180:220:0.3);

\draw[black, thick] (9,2.5) arc (0:40:0.3);
\draw[black, thick] (9,2.5) arc (0:-40:0.3);

\end{tikzpicture}
\caption{Maximal operator of the Cantor function.}
\label{Pic1}
\end{figure}
\end{center}

Consider $x = \frac{2}{3}, l = \frac{1}{3}, r = 1$. The average values of $h$ on $(l,x)$ and $(x,r)$ are respectively $h(x) = \frac{1}{2}$ and $\frac{h(x) + h(r)}{2} = \frac{3}{4}$ (see Figure~\ref{Pic1}). Thus, the average value on $(l, r)$ is $\frac{3 h(x) + h(r)}{4} = \frac{5}{8} = h(\frac{19}{27}) = h(\frac{8x + r}{9})$ which means that $M h(x) \geq h(\frac{8x+r}{9})$. Since $h, M h$ are increasing, we deduce that the whole interval $(h(x),h(\frac{8x+r}{9}))= (\frac{1}{2},\frac{5}{8})$ lies outside $h (C_{h})$.
More generally, take $x = \sum_{k=1}^\infty \frac{2a_k}{3^k} \in \mathcal C$ with $(a_k)_{k=1}^\infty \in \{0,1\}^{\mathbb N}$ of the form $(a_1, \dots, a_{K-1}, 1, 0, 0, \dots)$ for some arbitrarily chosen $K \in \mathbb N$ and $a_1, \dots, a_{K-1} \in \{0,1\}$. Let $l, r \in \mathcal C$ be the points corresponding to the sequences $(a_1, \dots, a_{K-1}, 0, 1, 1, \dots)$ and $(a_1, \dots, a_{K-1}, 1, 1, 1, \dots)$, respectively. Then the average values of $h$ on $(l,x)$ and $(x,r)$ are respectively $h(x) = \sum_{k=1}^\infty \frac{a_k}{2^k}$ and $\frac{1}{2}(h(x) + h(r)) = 2^{-K-1} + \sum_{k=1}^\infty \frac{a_k}{2^k}$. This means that $M h(x) \geq 2^{-K-2} + \sum_{k=1}^\infty \frac{a_k}{2^k} = h(\frac{8x+r}{9})$. We deduce that the interval $(h(x), h(\frac{8x+r}{9}))$ lies outside $h (C_{h})$, and observe that the point $\frac{8x+r}{9} \in \mathcal C$ corresponds to the sequence $(a_1, \dots, a_{K-1}, 1, 0, 0, 1, 1, \dots)$. Denote $\mathcal D = \{ i2^{j} : i, j \in \mathbb Z \}$ and for $y \in [0,1] \setminus \mathcal D$ let $(b^y_1,b^y_2,\dots) \in (0,1)^{\mathbb N}$ be the (uniquely determined) sequence of coefficients in the binary expansion of $y$, that is, $y = \sum_{k =1}^{\infty} \frac{b^y_k}{2^k}$. According to this notation, the interval $(h(x), h(\frac{8x+r}{9}))$ contains all points $y \in [0,1] \setminus \mathcal D$ such that $(b^y_1, \dots, b^y_{K+2}) = (a_1, \dots, a_{K-1},1,0,0)$. 
Thus, we conclude that $h (C_{h})$ contains no elements $y \in [0,1] \setminus \mathcal D$ with the pattern $(1,0,0)$ occurring in the corresponding sequences $(b^y_1, b^y_2,\dots)$. However, the set $\mathcal D$ is of Lebesgue measure zero and for almost all $y \in [0,1] \setminus \mathcal D$ there exists $K \in \mathbb N$ such that $(b^y_{K}, b^y_{K+1}, b^y_{K+2})=(1,0,0)$. Therefore, we have $m(h (C_{h})) = 0$, as desired.         

\section{Theorem \ref{maintheoremmultiple} for $n=1$} \label{S4}
We now prove a local version of Theorem~\ref{maintheoremmultiple} for $n=1$. We abbreviate $\mu_1, d_1, E_1$ to $\mu, d, E$. Given $\delta \in (0, \infty)$ we define the local Hardy--Littlewood maximal operator as follows: for each $x\in \R$, we let 
\[
M_\delta g(x) \coloneqq \underset{r \in (0, \delta]}{\sup}
\, \intav{B(x,r)} |g|.
\]
Denote also $C_f \coloneqq \{ x \in \R; M f(x) = f(x) \}$ and $C_{f, \delta} \coloneqq \{ x \in \R; M_\delta f(x) = f(x) \}$. Then, the following is true. 

\begin{proposition}\label{casen=1}
Let $f(x)=\mu([0,x])$ with $\mu$ being $d$-Ahlfors regular supported on a closed set $E\subset [0,1]$, with $d \in (0,1)$ and some constant $C \in (0,\infty)$. Then, for each $\delta \in (0,\infty)$ we have $m(f(C_{f,\delta})) = 0$. 
\end{proposition}

\noindent Notice that Proposition~\ref{casen=1} implies that Theorem~\ref{maintheoremmultiple} with $n=1$ is true. Indeed, we have $C_f \subset C_{f,1}$ so that $m(f(C_f)) \leq m(f(C_{f,1}))=0$ and the thesis follows by Proposition~\ref{justluzinproperty}.

Write $(0,1) \setminus E =\bigcup_{i \in \Lambda}(a_i,b_i)=\bigcup_{i \in \Lambda}(c_i-r_i,c_i+r_i)$ using some countable set $\Lambda \subset \N$. Given $\delta \in (0, \infty)$ and any interval $I \subset \R$, we have $m(f(C_{f,\delta} \cap I)) \leq m(f(I)) = 0$ whenever $I \cap E = \emptyset$. Since $[0,1]$ can be covered by finitely many intervals $I$ of length at most $\delta$, to prove Proposition~\ref{casen=1} it suffices to show $m(f(C_{f,\delta} \cap I)) = 0$ for each such $I$.  


The following lemma is crucial for our purposes. 
\begin{lemma}\label{L1}
	For an open interval $J\subset [0,1]$, let $\Lambda_J \coloneqq \{i \in \Lambda; (a_i, b_i) \subset J\}$. Then
	\[
	\mu \Big( \bigcup_{i \in \Lambda_J} (b_{i},b_{i}+r_{i})\cap J \Big)\ge \frac{1}{2} C^{-2}4^{-d} \mu(J).
	\]
\end{lemma}

\begin{proof}
	Let us write $\overline{J}=[\alpha,\beta]$.
	If $\mu(J) = 0$, then the thesis is trivial. If $\mu(J) > 0$, then we can find $\tilde J \subset J$ whose endpoints belong to $E$ and such that $\mu(\tilde J) = \mu(J)$. Thus, we can assume that the endpoints of $J$ belong to $E$. 
	
	First, if $b_{i} + r_{i} > \beta$ for some $i \in \Lambda_J$, then
	$
	\mu((b_{i}, b_{i} + r_{i}) \cap (a_{i}, \beta))=\mu((a_{i}, \beta))
	\geq \tfrac{1}{2} C^{-2}4^{-d} \mu((a_{i}, \beta)).
	$
	Thus, it suffices to verify the thesis with $J$ replaced by $(\alpha, \tilde \beta)$, where $ \tilde \beta = \min_{i \in \Lambda_J; b_i + r_i > \beta} a_i \in E$ (in this case, it is easy to show that the smallest $a_i$ indeed exists). We repeat this procedure recursively until it stabilizes (after finitely or infinitely many steps) at some point $\beta_0 \in [\alpha, \beta]$. If $\beta_0 = \alpha$ or $\mu((\alpha, \beta_0))=0$, then we obtain the thesis right away. Otherwise, we end up with $(\alpha, \beta_0) \subset J$ whose endpoints belong to $E$ (if there was infinitely many steps, then we use the fact that $E$ is closed), satisfying $\mu((\alpha, \beta_0)) >0$ and $b_i + r_i \leq \beta_0$ for all $i \in \Lambda_{(\alpha, \beta_0)}$. For $(\alpha, \beta_0)$, called $J$ for simplicity later on, the proof goes as follows. 
	
	Given $J$ as above we choose $\Lambda_J^{\circ} \subset \Lambda_J$ with the Besicovitch covering properties $\bigcup_{i \in \Lambda_J} (b_i - r_i, b_i + r_i) \subset \bigcup_{i \in \Lambda^{\circ}_J} (b_i - r_i, b_i + r_i)$ and $\| \sum_{i \in \Lambda^{\circ}_J} \chi_{(b_i - r_i, b_i + r_i)} \|_{L^\infty} \leq 2$. Then
	\begin{align*}
	\mu \Big(\bigcup_{i \in \Lambda_J} (b_{i},b_{i}+r_{i})\cap J\Big)
	&\geq \mu \Big(\bigcup_{i \in \Lambda^{\circ}_J} (b_{i},b_{i}+ {r}_{i}) \Big)
	= \mu \Big(\bigcup_{i \in \Lambda^{\circ}_J} (b_{i} - {r}_{i},b_{i}+ {r}_{i}) \Big)\\
	&\geq \frac{1}{2}\sum_{i \in \Lambda^{\circ}_J}\mu((b_i,b_i+r_i))\geq \frac{1}{2C} \sum_{i \in \Lambda^{\circ}_J} {r}_{i}^d
	\geq \frac{1}{2C} \Big(\sum_{i \in \Lambda^{\circ}_J} {r}_{i}\Big)^d.
	\end{align*}
	Also, we have
	\begin{align*}
	\sum_{i \in \Lambda^{\circ}_J} {r}_{i}
	\geq \frac{1}{2} m \Big( \bigcup_{i \in \Lambda^{\circ}_J} (b_i - r_i, b_i + r_i) \Big)
	&= \frac{1}{2} m \Big( \bigcup_{i \in \Lambda_J} (b_i - r_i, b_i + r_i) \Big)
	\\
	&\geq \frac{1}{2} m \Big( \bigcup_{i \in \Lambda^{\circ}_J} (c_i, b_i) \Big)
	= \frac{1}{4} m(J).
	\end{align*}
	Combining both estimates completes the proof, since
	\[
	\mu \Big(\bigcup_{i \in \Lambda_J} (b_{i},b_{i}+r_{i})\cap J \Big)
	\geq \frac{1}{2C} \Big(\sum_{i \in \Lambda^{\circ}_J} {r}_{i}\Big)^d
	\geq \frac{1}{2C} \Big(\frac{m(J)}{4}\Big)^d \geq \frac{\mu(J)}{2C^2 4^d}, 
	\]
	where the last inequality follows because $\alpha \in E$ and $J \subset (\alpha - m(J), \alpha + m(J))$. 
\end{proof}
The next lemma provides a disjoint version of the previous one.
\begin{lemma}\label{L2}
	Let $J, \Lambda_J$ be as in Lemma~\ref{L1} and for each $i \in \Lambda_J$ set $\tilde r_i \coloneqq \min\{r_i, \beta - b_i\}$. Then
	$\mu \big(\bigcup_{i \in \Lambda_J^\circ} (b_{i} - \tilde r_{i},b_{i}+\tilde r_{i}) \big) \ge \frac{1}{2} C^{-4}12^{-d} \mu(J)$ holds for some $\Lambda_J^\circ \subset \Lambda_J$ such that the intervals $(b_i, b_i + \tilde r_i)$, $i \in \Lambda_J^\circ$, are disjoint.
\end{lemma}

\begin{proof} 
	Assuming $\mu(J) > 0$ we choose $\Lambda_J^\circ \subset \Lambda_J$ with the Vitali covering properties 
	$\bigcup_{i \in \Lambda_J} (b_i - \tilde r_i, b_i + \tilde r_i) \subset \bigcup_{i \in \Lambda^{\circ}_J} (b_i - 3 \tilde r_i, b_i + 3 \tilde r_i)$ and $\| \sum_{i \in \Lambda^{\circ}_J} \chi_{(b_i - \tilde r_i, b_i + \tilde r_i)} \|_{L^\infty} = 1$.
	Then, by applying Lemma~\ref{L1},
	\[
	\frac{\mu(J)}{2C^2 4^d}
	\le \mu \Big(\bigcup_{i \in \Lambda_J}^{\infty} (b_{i},b_{i}+ \tilde r_{i}) \Big)
	\le \mu \Big(\bigcup_{i \in \Lambda^{\circ}_J}^{\infty} (b_{i}-3 \tilde r_{i},b_{i}+3 \tilde r_{i}) \Big)
	\le C^23^d \mu \Big(\bigcup_{i \in \Lambda^{\circ}_J}^{\infty} (b_{i}-\tilde r_{i},b_{i}+ \tilde r_{i}) \Big),
	\]
	where in the last inequality disjointness and the Ahlfors regularity condition were used.
\end{proof}
In the next lemma, properties of $M_\delta$ are used to construct gaps in the image of $C_{f,\delta}.$ We write a more general version of what is needed to prove Proposition~\ref{casen=1}, having in mind Theorem~\ref{maintheoremmultiple}.
\begin{lemma}\label{detachment2}
Let $J, \Lambda_J$ be as in Lemma~\ref{L1}, and assume that $m(J) \leq \delta$. 
Moreover, suppose that
\[
\mu([b_i,b_i+r_i]) \ge 4 \eta([b_i-2r_i,b_i])
\] 
for all $i \in \Lambda_J$, where $\mu$ is as before and $\eta$ is some other finite nonnegative Borel measure supported on $[0,1]$. Then, for the function $g(x) = (\mu+\eta)([0,x])$, we have  
\[
\bigcup_{i \in \Lambda_J} \big(g(b_i),g(b_i)+ \tfrac{1}{8} \mu ([b_i,b_i+r_i])\big) \subset g(\{x\in J; M_{\delta} g(x)=g(x)\})^{\sf c}.
\]
\end{lemma}
\begin{proof}
Observe that for any $i \in \Lambda_J$ we have $r_i \leq \frac{\delta}{2}$, since $m(J) \leq \delta$. Thus,
\[
M_\delta g(b_i) \ge \intav{[b_i-2r_i,b_i+2r_i]}g\\
\ge \frac{\int_{[b_i-2r_i,b_i]}g}{4r_i}+\frac{\int_{[b_i,b_i+2r_i]}g}{4r_i}.
\]
Next, notice that 
\[
\frac{\int_{[b_i,b_i+2r_i]}g}{4r_i}\ge \frac{g(b_i)}{4}+\frac{g(b_i+r_i)}{4}\ge \frac{g(b_i)}{4}+\frac{g(b_i)+\mu([b_i,b_i+r_i])}{4}.
\]
On the other hand, by our assumption,
\[
\frac{\int_{[b_i-2r_i,b_i]}|g|}{4r_i}
\ge \frac{g(b_i-2r_i)}{2}=\frac{g(b_i)}{2}-\frac{\eta([b_i-2r_i,b_i])}{2}
\ge \frac{g(b_i)}{2} - \frac{\mu[b_i,b_i+r_i]}{8} .
\] 
Combining these inequalities gives
\begin{align*}
M_{\delta}g(b_i)\ge 
g(b_i)+\tfrac{1}{8} \mu([b_i,b_i+r_i]),    
\end{align*}
from where the thesis follows, since both $g, M_\delta g$ are increasing.
\end{proof}


\begin{lemma}\label{L5}
 Let $J, \Lambda_J^\circ, \tilde r_i$ be as in Lemma~\ref{L2}. Then, the intervals $\big(f(b_i), f(b_i)+\frac{1}{8}\mu((b_i,b_i + \tilde r_i))\big)$, $i \in \Lambda_J^\circ$, are disjoint and contained in $f(J)$.
\end{lemma}
\begin{proof}
The second part is clear because each interval $(b_i, b_i + \tilde r_i)$ is contained in $J$. Regarding the first part,	observe that given $i_1, i_2 \in \Lambda_J^\circ$, if $b_{i_1} < b_{i_2}$, then 
$
f(b_{i_2})-f(b_{i_1})= \mu((b_{i_1}, b_{i_2})) \ge \mu((b_{i_1}, b_{i_1} + \tilde r_{i_1})),
$ 
since the intervals $(b_{i}, b_{i} + \tilde r_{i})$, $i \in \Lambda_J^\circ$, are disjoint. The measure of the last interval is positive in view of the regularity condition and $\mu((a_{i_1}, b_{i_1}))=0$. Thus,
$f(b_{i_2}) > f(b_{i_1}) + \frac{\mu((b_{i_1}, b_{i_1} + \tilde r_{i_1}))}{8}$, 
as desired.
\end{proof}

\begin{proof}[Proof of Proposition~\ref{casen=1}]
We take an interval $I$ such that $m(I) \leq \delta$, and show $m(f(C_{f,\delta} \cap I)) = 0$.
 
First, observe that $m(f(C_{f,\delta} \cap I)) \leq m(f(I)) = \mu(I)$. Denote $K \coloneqq \frac{1}{32}C^{-4}12^{-d} \in (0,1)$. By Lemma~\ref{L2} (with $J = I$) we can find a finite set $\Lambda^\circ \subset \Lambda$ such that $(b_i,b_i+\tilde{r_i})$, $i \in \Lambda^\circ$, are disjoint, and $\sum_{i \in \Lambda^\circ} \mu((b_i,b_i+\tilde r_i)) \ge 8K \mu(I)$.
Then by Lemmas~\ref{detachment2} (with $J = I$ and $\eta = 0$) and~\ref{L5} (with $J = I$) we obtain
\[
B_i \coloneqq \big( f(b_i), f(b_i)+\tfrac{1}{8}\mu((b_i,b_i+\tilde r_i)) \big) \subset f(I) \setminus f(C_{f,\delta} \cap I)
\]
and, consequently,
\[
m \Big(\bigcup_{i \in \Lambda^\circ} B_i \cap f(I) \Big) \ge \sum_{i \in \Lambda^\circ} \tfrac{1}{8} \mu((b_i,b_i+\tilde r_i)) \ge K\mu(I) = K m(f(I)).
\]
We let $I_{1},\dots, I_{N_1}$ be the intervals that compose $[0,1] \setminus \cup_{i \in \Lambda^\circ} f^{-1}(B_i)$. Then, since $f$ is increasing,
\[
m ( f(C_{f,\delta} \cap I) ) \leq \sum_{n_1=1}^{N_1} m (f(I_{n_1} )) \le (1-K) m(f(I)).
\]
Next, we proceed analogously replacing $I$ with each of the sets $I_{n_1}$. As a result we obtain  
\[
m(f(C_{f,\delta} \cap I)) \leq \sum_{n_1=1}^{N_1} \sum_{n_2 = 1}^{N_{2}(n_1)} m (f(I_{n_1,n_2})) \leq \sum_{n_1=1}^{N_1} (1-K) m(f(I_{n_1})) \le (1-K)^2 m(f(I))
\]
with certain appropriately constructed intervals $I_{n_1,n_2} \subset I_{n_1}$. Repeating this procedure recursively, we deduce that $m(f(C_{f,\delta} \cap I)) \le (1-K)^L m(f(I))$ for each $L \in \mathbb N$ so that $m(f(C_{f,\delta} \cap I)) = 0$ indeed holds.
\end{proof}

\section{Proof of Theorem \ref{maintheoremmultiple}} \label{S5}
We proceed inductively to prove Theorem~\ref{maintheoremmultiple} in the general case. Precisely, Theorem~\ref{maintheoremmultiple} is a consequence of Proposition~\ref{casen=1}, corresponding to $n=1$, and the following result which enables the inductive step. 
\begin{proposition} \label{induction}
	For $n \in \N$, let $\{\mu_i\}_{i=1}^{n+1}$ be a collection of $d_i$-Ahlfors regular measures, $0 < d_{n+1} < d_n < \cdots < d_1 < 1$, associated with closed sets $E_i\subset [0,1]$. Assume that for $f_n(x) = (\mu_1 + \cdots + \mu_n)([0,x])$ and all $\delta \in (0, \infty)$ one has $m(f_n(C_{f_n, \delta})) = 0$. Then for $f_{n+1}(x) = (\mu_1 + \cdots + \mu_{n+1})([0,x])$ and all $\delta \in (0, \infty)$ one has $m(f_{n+1}(C_{f_{n+1}, \delta})) = 0$. In particular, $m(f_{n+1}(C_{f_{n+1}})) = 0$ and Theorem~\ref{maintheoremmultiple} holds for $f = f_{n+1}$ by Proposition~\ref{justluzinproperty} applied with $g = f_{n+1}$.
\end{proposition}

Before the proof, several remarks are in order.

\begin{enumerate}[label=(\alph*)]
	\item \label{A} If $\mu, \mu'$ are $d$-Ahlfors measures associated with closed sets $E, E' \subset [0,1]$, then $\mu + \mu'$ defines a $d$-Ahlfors measure associated with a closed set $E \cup E' \subset [0,1]$. Because of this to prove Theorem~\ref{maintheoremmultiple} it suffices to show Proposition~\ref{induction} where all parameters $d_i$ are different.
	\item \label{B} The thesis is the stronger the smaller $\delta$ is. Indeed, if $0 < \delta' < \delta < \infty$, then $C_{f,\delta} \subset C_{f,\delta'}$. 
	\item \label{C} If the hypothesis of Proposition~\ref{induction} holds, then for any interval $I = (\alpha, \beta)$ such that $I \cap E_{n+1} = \emptyset$ one has $m(f_{n+1}(C_{f_{n+1},\delta} \cap I))=0$ for all $\delta \in (0,\infty)$. Indeed, for $\delta \in (0,\infty)$ choose any $L \in \N$ such that $L^{-1} < \delta$ and $2L^{-1} < m(I)$, and denote $I_L \coloneqq (\alpha + L^{-1}, \beta - L^{-1})$. Then
	\[
	f_{n+1}(C_{f_{n+1},\delta} \cap I_L) \subset f_{n+1}(C_{f_{n+1},L^{-1}} \cap I_L) = f_{n}(C_{f_{n},L^{-1}} \cap I_L) + c_{I}
	\subset f_{n}(C_{f_{n},L^{-1}}) + c_{I}
	\]
	for some constant $c_I$, so that all these sets are $m$-null sets. Letting $L \to \infty$, we obtain the thesis. This observation tells us that it suffices to analyze the image of the set $C_{f_{n+1}, \delta} \cap E_{n+1}$. Also, we see that dealing with local operators $M_\delta$ in Proposition~\ref{induction} simplifies the analysis outside $E_{n+1}$.       
	\item \label{D} Since $d_{n+1}$ is strictly smaller than $d_1, \dots, d_n$, there exists $\delta_0$ such that, for $x \in E_{n+1}$ and $r \leq \delta_0$,
	\[
	\mu_{n+1}([x-r,x+r]) \ge 4 (\mu_1 + \dots + \mu_n)([x-2r,x+2r]).
	\]
	In other words, for $x \in E_{n+1}$ and small radii the impact of the measures $\mu_1, \dots, \mu_{n}$ is negligible. This is a consequence of the Ahlfors regularity of the measures. 
	This is also the second reason why looking at local operators $M_\delta$ simplifies the analysis. Notice that by \ref{B} it suffices to show that the thesis holds for all $\delta \in (0, \delta_0]$.      
\end{enumerate}

\begin{proof}[Proof of Proposition~\ref{induction}]
	Let $\delta \in (0, \delta_0]$. We take an interval $I = (\alpha, \beta)$ such that $m(I) \leq \delta$, and show $m(f_{n+1}(C_{f_{n+1},\delta} \cap I)) = 0$. In what follows, we put $\mu = \mu_{n+1}$ and $\eta = \mu_1 + \dots + \mu_n$. Also, we adopt the rest of notation from Section~\ref{S4} with $\mu_{n+1}$ playing the role of $\mu$ (that is, $a_i, b_i, \Lambda$ refer to $\mu_{n+1}$).
	
	We first show that the thesis holds if the following claims are true.
	
	\smallskip \noindent {\bf Claim~1.} For each interval $J \subset I$ and each $\epsilon \in (0,\infty)$ it is possible to find finitely many disjoint open intervals $J_1, \dots, J_N \subset J$ such that $\mu(J) = \mu(J_1 \cup \cdots \cup J_N)$ and $\eta(J_1 \cup \cdots \cup J_N) \leq \epsilon$.
	
	\smallskip \noindent {\bf Claim~2.} For $K \coloneqq \frac{1}{64}C_{n+1}^{-4}12^{-d_{n+1}} \in (0,1)$ and each interval $J \subset I$ it is possible to find finitely many disjoint open intervals $J^1, \dots, J^N \subset J$ such that $\mu(J^1 \cup \cdots \cup J^N) \geq K \mu(J)$ and $f_{n+1}(J^1 \cup \cdots \cup J^N) \subset f_{n+1}(J) \setminus f_{n+1}(C_{f_{n+1}, \delta} \cap J)$.
	
	\smallskip Indeed, take $\epsilon_1 \in (0,\infty)$ and let $I_1 \cup \cdots \cup I_N \cup \tilde I_1 \cup \cdots \cup \tilde I_{\tilde N}$ be a disjoint sum covering $I$ up to a finite set of endpoints, with $I_1, \dots, I_N$ being $J_1, \dots, J_N$ from {\bf Claim~1} applied with $J = I$ and $\epsilon = \epsilon_1$. Since $\mu(\tilde I_i) = 0$ and $\tilde I_i$ is open, we have $\tilde I_i \cap E_{n+1} = \emptyset$ so that $m(f_{n+1}(C_{f_{n+1},\delta} \cap \tilde I_i)) = 0$ holds by the inductive hypothesis, see \ref{C}. Then, for each $i \in \{1, \dots, N\}$ let $I_i^{1} \cup \cdots \cup I_i^{N_i} \cup \tilde I_i^1 \cup \cdots \cup \tilde I_i^{\tilde N_i}$ be a disjoint sum covering $I_i$ up to a~finite set of endpoints, with $I_i^1, \dots, I_i^{N_i}$ being $J^1, \dots, J^N$ from {\bf Claim~2} applied with $J = I_i$. Then
	\[
	m(f_{n+1}(C_{f_{n+1}, \delta} \cap I)) = \sum_{i=1}^N m(f_{n+1}(C_{f_{n+1}, \delta} \cap I_i))
	= \sum_{i=1}^N \sum_{i' = 1}^{\tilde N_i} m(f_{n+1}(C_{f_{n+1}, \delta} \cap \tilde I_i^{i'}))
	\leq \sum_{i=1}^N \sum_{i' = 1}^{\tilde N_i} m(f_{n+1}(\tilde I_i^{i'}))
	\]  
	and the last sum is bounded by
	\[
	 \sum_{i=1}^N \sum_{i' = 1}^{\tilde N_i} (\mu + \eta)(\tilde I_i^{i'}) 
	 \leq
	 \sum_{i=1}^N \sum_{i' = 1}^{\tilde N_i} \mu(\tilde I_i^{i'})
	 + \sum_{i=1}^N \eta(I_i)
	 \leq \sum_{i=1}^N (1-K)\mu(I_i) + \epsilon_1 = (1-K)\mu(I) + \epsilon_1.
	\]
	However, instead of estimating $m(f_{n+1}(C_{f_{n+1}, \delta} \cap \tilde I_i^{i'}))$ by $m(f_{n+1}(\tilde I_i^{i'}))$, we may apply the same procedure as before with some $\epsilon_2 \in (0, \infty)$ and $\tilde I_i^{i'}$ in place of $I$. Then
	\[
	m(f_{n+1}(C_{f_{n+1}, \delta} \cap \tilde I_i^{i'})) 
	= \sum_{j=1}^{N_{i,i'}} \sum_{j' = 1}^{\tilde N_{i,i',j}}
	m(f_{n+1}(C_{f_{n+1}, \delta} \cap \tilde I_{i,j}^{i',j'}))
	\leq (1-K) \mu(\tilde I_i^{i'})) + \epsilon_2 
	\]
	for certain appropriately chosen intervals $\tilde I_{i,j}^{i',j'} \subset I_i^{i'}$,
	and summing this over $i,i'$ gives
	\[
	m(f_{n+1}(C_{f_{n+1}, \delta} \cap I)) \leq
	\sum_{i=1}^N \sum_{i' = 1}^{\tilde N_i} \big( (1-K) \mu(\tilde I_i^{i'}) + \epsilon_2 \big) \leq (1-K)^2 \mu(I) + \tfrac{1}{2} \epsilon_1, 
	\]
	provided that $(\tilde N_1 + \dots + \tilde N_N) \epsilon_2 \leq \frac{1}{2} \epsilon_1$. We then deduce that $m(f_{n+1}(C_{f_{n+1}, \delta} \cap I)) \leq (1-K)^L \mu(I) + \frac{1}{L} \epsilon_1$ for each $L \in \N$, repeating this procedure recursively. Thus, $m(f_{n+1}(C_{f_{n+1}, \delta} \cap I)) = 0$ follows. 
	
	\smallskip \noindent {\bf Proof of Claim~1.} Take $L \in \N$ and split $J$ into $L$ intervals $J'$ of equal lengths $m(J') = L^{-1} m(J)$. Let $J_1, \dots, J_N$ be those of them for which $\mu(J') > 0$ holds. We shall show that $\eta(J_1 \cup \cdots \cup J_N) \leq \epsilon$ is satisfied for some large $L$. 
	To this end, note that for a given parameter $\kappa \in (0, \infty)$ one has $\eta(J_i) \leq \kappa \mu(3J_i)$ for each $i \in \{1, \dots, N\}$, provided that $L$ is large enough (here $3J_i$ stands for the open interval concentric with $J_i$ but $3$ times longer). Indeed, this follows, since $J_i \cap E_{n+1} \neq \emptyset$ and $d_{n+1}$ is the smallest Ahlfors parameter. Consequently, $\eta(J_1 + \dots + J_{N}) \leq 3 \kappa \mu(3J) \leq \epsilon$ if $\kappa$ is small enough, cf.~\ref{D}.      
	
	\smallskip \noindent {\bf Proof of Claim~2.}
	By Lemma~\ref{L2} we can find a finite set $\Lambda^\circ \subset \Lambda_J$ such that $(b_i,b_i+\tilde{r_i})$, $i \in \Lambda^\circ$, are disjoint, and $\sum_{i \in \Lambda^\circ} \mu((b_i,b_i+\tilde r_i)) \ge 16K \mu(J)$.
	Since $m(J) \leq m(I) \leq \delta \leq \delta_0$, we have
	\[
	\mu([b_i, b_i+r_i]) = \mu([b_i-r_i,b_i+r_i]) \ge 4 \eta ([b_i-2r_i,b_i+2r_i]) \geq 4 \eta ([b_i-2r_i,b_i])
	\]
	for each $i \in \Lambda_J$. Thus, by Lemmas~\ref{detachment2}~and~\ref{L5} we obtain
	\[
	\big( f_{n+1}(b_i), f_{n+1}(b_i)+\tfrac{1}{8}\mu((b_i,b_i+\tilde r_i)) \big) \subset f_{n+1}(J) \setminus f_{n+1}(C_{f_{n+1},\delta} \cap J).
	\]
	Let $\tilde r_i' \in (0, \tilde r_i)$ be such that
	$(\mu + \eta)((b_i, b_i + \tilde r_i')) = \frac{1}{8} \mu((b_i, b_i + \tilde r_i))$. Then, for $J^i = (b_i, b_i + \tilde r_i')$ we have
	$f_{n+1}(J^i) \subset f_{n+1}(J) \setminus f_{n+1}(C_{f_{n+1}, \delta} \cap J)$, as desired. Moreover, $\mu(J^i) \geq \frac{1}{2} (\mu + \eta)(J^i) = \frac{1}{16} \mu((b_i,b_i + \tilde r_i))$ by \ref{D}, since $\tilde r_i' \leq \delta_0$.  
	Consequently,
	$
	\sum_{i \in \Lambda^\circ} \mu(J^i) \ge \sum_{i \in \Lambda^\circ} \frac{1}{16} \mu((b_i,b_i+\tilde r_i)) \ge K \mu(J)
	$
	and we are done.
	
	\smallskip \noindent The proof of Proposition~\ref{induction} (and by \ref{A} and Proposition~\ref{justluzinproperty} also the proof of Theorem~\ref{maintheoremmultiple}) is complete.
\end{proof}

We conclude the article noting that repeating the arguments from Propositions~\ref{casen=1}~and~\ref{induction} one can also prove the following version of Theorem~\ref{maintheoremmultiple}.

\begin{corollary}
	For a given interval $I = (\alpha, \beta) \subset \R$ (we allow $\alpha=-\infty$ or $\beta = \infty$) define the associated operator $M_I$ as follows: for each $x \in I$, we let
	\[
	M_I g(x) \coloneqq \underset{r \in (0, \infty) : (x-r, x+r) \subset I}{\sup}
	\, \intav{B(x,r)} |g|.
	\]
	Denote also $C_{f, I} \coloneqq \{ x \in I; M_I f(x) = f(x) \}$ for $f$ from Theorem~\ref{maintheoremmultiple}. Then $m(f(C_{f,I})) = 0$ and hence $M_I f$ is absolutely continuous on $I$ ($M_I f$ is continuous and of bounded variation by the results stated in \cite{Kurka2010}).
\end{corollary}

\bibliography{Reference}

\providecommand{\bysame}{\leavevmode\hbox to3em{\hrulefill}\thinspace}
\providecommand{\MR}{\relax\ifhmode\unskip\space\fi MR }
\providecommand{\MRhref}[2]{%
  \href{http://www.ams.org/mathscinet-getitem?mr=#1}{#2}
}
\providecommand{\href}[2]{#2}
\begin{thebibliography}{10}

\bibitem{AP2007}
J.~M. Aldaz and J.~P\'erez~L\'azaro, \emph{Functions of bounded variation, the
  derivative of the one dimensional maximal function, and applications to
  inequalities}, Trans. Amer. Math. Soc. \textbf{359} (2007), no.~5,
  2443--2461. \MR{2276629}

\bibitem{CM2015}
E.~Carneiro and J.~Madrid, \emph{Derivative bounds for fractional maximal
  functions}, Trans. Amer. Math. Soc. \textbf{369} (2017), no.~6, 4063--4092.
  \MR{3624402}

\bibitem{CMP2017}
E.~Carneiro, J.~Madrid, and L.~B. Pierce, \emph{Endpoint {S}obolev and {BV}
  continuity for maximal operators}, J. Funct. Anal. \textbf{273} (2017),
  no.~10, 3262--3294. \MR{3695894}

\bibitem{grcenteredcontinuity}
C.~González-Riquelme, \emph{Continuity for the one-dimensional centered
  {H}ardy-{L}ittlewood maximal operator at the derivative level}, preprint,
  arxiv.org/abs/2109.09691, 2021.

\bibitem{GRK}
C.~González-Riquelme and D.~Kosz, \emph{{BV} continuity for the uncentered
  {H}ardy–{L}ittlewood maximal operator}, J. Funct. Anal. \textbf{281}
  (2021), no.~2, 109037.

\bibitem{HKKT2015}
T.~Heikkinen, J.~Kinnunen, J.~Korvenp{\"a}{\"a}, and H.~Tuominen,
  \emph{Regularity of the local fractional maximal function}, Ark. Mat.
  \textbf{53} (2015), no.~1, 127--154.

\bibitem{Kinnunen1997}
J.~Kinnunen, \emph{The {H}ardy-{L}ittlewood maximal function of a {S}obolev
  function}, Israel J. Math. \textbf{100} (1997), 117--124. \MR{1469106}

\bibitem{Kurka2010}
O.~Kurka, \emph{On the variation of the {H}ardy-{L}ittlewood maximal function},
  Ann. Acad. Sci. Fenn. Math. \textbf{40} (2015), no.~1, 109--133. \MR{3310075}

\bibitem{lahti}
P.~Lahti, \emph{On the regularity of the maximal function of a {BV} function},
  preprint, arxiv.org/abs/2007.05752, to appear in J. Differential Equations,
  2020.

\bibitem{Luiro2017}
H.~Luiro, \emph{The variation of the maximal function of a radial function},
  Ark. Mat. \textbf{56} (2018), no.~1, 147--161. \MR{3800463}

\bibitem{RSW}
J.P.G. Ramos, O.~Saari, and J.~Weigt, \emph{Weak differentiability for
  fractional maximal functions of general ${L}^p$ functions on domains}, Adv.
  in Math. \textbf{368} (2020), 107144.

\bibitem{Tanaka2002}
H.~Tanaka, \emph{A remark on the derivative of the one-dimensional
  {H}ardy-{L}ittlewood maximal function}, Bull. Austral. Math. Soc. \textbf{65}
  (2002), no.~2, 253--258. \MR{1898539}

\bibitem{charweigt}
J.~Weigt, \emph{Variation of the uncentered maximal characteristic function},
  preprint, arxiv.org/abs/2004.10485, to appear in Rev. Mat. Iberoam., 2020.

\bibitem{weigt2021variation}
J.~Weigt, \emph{The variation of the uncentered maximal operator with respect
  to cubes}, preprint, arxiv.org/abs/2109.10747, 2021.

\end{thebibliography}
\bibliographystyle{amsplain} 
\end{document}